\documentclass[12pt]{amsart}
\usepackage[dvips]{graphicx}
\usepackage{amssymb,amsmath,amsthm}
\usepackage[all]{xy}

\def\NZQ{\mathbb}        

\def\RR{{\NZQ R}}

\def\P{\mathcal{P}}
\def\H{\mathcal{H}}
\def\frk{\mathfrak}        

\def\Phi{{\frk N}}

\newtheorem{Theorem}{Theorem}[section]
\newtheorem{Lemma}[Theorem]{Lemma}

\newtheorem{Proposition}[Theorem]{Proposition}

\theoremstyle{definition}

\newtheorem{Example}[Theorem]{Example}

\begin{document}
\title{Decomposable edge polytopes of finite graphs}
\author{Atsushi Funato, Nan Li, Akihiro Shikama}
\address{Atsushi Funato, Department of Pure and Applied Mathematics, Graduate School of Information Science and Technology,
Osaka University, Toyonaka, Osaka 560-0043, Japan}
\email{a-funato@cr.math.sci.osaka-u.ac.jp}

\address{Nan Li, Department of Mathematics,
Massachusetts Institute of Technology,
Cambridge, MA 02139, USA}
\email{nan@math.mit.edu}

\address{Akihiro Shikama, Department of Pure and Applied Mathematics, Graduate School of Information Science and Technology,
Osaka University, Toyonaka, Osaka 560-0043, Japan}
\email{a-shikama@cr.math.sci.osaka-u.ac.jp}

\thanks{
\hspace{5.5mm}{\bf Keywords:}
finite connceted graphs, edge polytopes, separating hyperplane.\\
}

\begin{abstract}
Edge polytopes is a class of interesting polytope with rich algebraic and combinatorial properties, 
which was introduced by Ohsugi and Hibi. In this papar, we follow a previous study on cutting edge polytopes
by Hibi, Li and Zhang. Instead of focusing on the algeraic properties of the subpolytopes as the previous
study, in this paper, we take a closer look on the graphs whose edge polytopes are decomposable. In particular,
we answer two important questions raised in the previous study about 1) the relationship between type I and type II
decomposable graphs and 2) description of decomposable graphs in terms of the underlying graphs.
\end{abstract}
\maketitle
\section*{Introduction}

Let $\P$ be an integral polytope  in $\RR^d$.
Here we say a polytope is integral if all vertices of the polytope are integer points.
We say that $\P$ is {\em decomposable} if there exist a hyperplane $\H$ with $\H \cap (\P \setminus \partial \P )\neq\emptyset$ such that each of the convex polytopes $\P\cap \H^{(+)}$ and $\P\cap \H^{(-)}$ is integral.
In this paper, we discuss decomposability of a special class of integral polytopes, called {\em edge polytopes}. Edge polytopes are introduced by Ohsugi and Hibi in \cite{OH2}. Edge polytopes are integral polytopes arising from finite connected graphs.
Let $G$ be a finite simple graph with vertex set $V=[d]=\{1, \ldots , d\}$ and edge set $E(G)=\{e_{1}, \ldots , e_{n}\}$. Let $\mbox{\boldmath $e$}_{i}$ be the $i$-th unit coordinate vecter of the Euclidean space ${\RR}^{d}$. If $e=(i,  j)$ is an edge of $G$, then  we set $\rho(e)=\mbox{\boldmath $e$}_{i}+\mbox{\boldmath $e$}_{j}\in {\RR}^{d}$. The edge polytope $\P_G$ of $G$ is the convex hull of $\{\rho(e_{1}),\ldots \rho(e_{n})\}$ in ${\RR}^{d}$. 

The only integer point belongs to an edge polytope are its vertices. It follows from this fact that 
an edge polytope is decomposable if and only if there exist a hyperplane $\H$ which is not 
supporting hyperplane of $\P_G$ and such that for each edge $E$ of $\P_G$ with $\H \cap E \neq \emptyset$ it follows that $E \subset \H$ or $\H \cap E$ is an end point of $E$.

In \cite{HLZ}, type I and II decomposability for edge polytopes are introduced and an algorithm to decide decomposability is given. 
In this paper, we discuss decomposability of edge polytopes of finite connected simple graphs. We carry out the study of edge polytope decomposability and take a closer look at the graphs with type I and type II decomposable edge polytopes. 
In Section 2, we study decomposability of edge polytopes in terms of the underlying graphs.



\section{Edge polytopes and its decomposablity}
Recall that a convex polytope is integral if all of its vertices have integral coordinates: in particular, $\P_G$ is an integral polytope. Let $\partial \P$ denote the boundary of a polytope $\P$. We say that $\P$ is decomposable if there exists a hyperplane $\H$ of ${\RR}^{d}$ with $\H\cap (\P\backslash\partial \P)\neq\emptyset$ such that each of the convex polytopes $\P\cap \H^{(+)}$ and $\P\cap \H^{(-)}$ is integral.  Here $\H^{(+)}$ and $\H^{(-)}$ are the closed half-space of ${\RR}^{d}$ with $\H^{(+)}\cap \H^{(-)}=\H$. Such a hyperplane $\H$ is called a {\em separating hyperplane} of $\P$.  We say a graph $G$ is {\it decomposable} if edge polytope $\P_G$ of $G$ is decomposable.
A simple graph is a graph with no loops and no multiple edges. Let $G$ be a finite connected simple graph with vertex set $V=[d]=\{1, \ldots , d\}$ and edge set $E(G)=\{e_{1}, \ldots , e_{n}\}$. Let $\mbox{\boldmath $e$}_{i}$ be the $i$-th unit coordinate vecter of the Euclidean space ${\RR}^{d}$ . If $e=(i, j)$ is an edge of $G$, then  we set $\rho(e)=\mbox{\boldmath $e$}_{i}+\mbox{\boldmath $e$}_{j}\in {\RR}^{d}$. The edge polytope $\P_G$ of $G$ is the convex hull of $\{\rho(e_{1}),\ldots \rho(e_{n})\}$ in ${\RR}^{d}$. \\

Since edge polytope $\P_G$ is a (0,1) polytope, the only interger points in $\P_G$ are its vertices.
The vertices of the edge polytope $\P_G$ of $G$ are  $\{\rho(e_{1}),\ldots , \rho(e_{n})\}$, but not all edges of the form $(\rho(e_{i}), \rho(e_{j}))$ actually occur. 
In the recent research, the number of edges of edge polytopes has been discussed (\cite{HMOS}).
For $i\neq j$, let co$(e_{i}, e_{j})$ be the convex hull of the pair of $\{\rho(e_{i}), \rho(e_{j})\}$. The edges of $\P_G$ will be a subset of these co$(e_{i}, e_{j})$s. For edges $e=(i, j)$ and $f=(k, \ell)$, call the pair of edges $(e, f)$ {\em cycle-compatible} with $C$ if there exists a 4-cycle $C$ in the subgraph of $G$ induced by $\{i, j, k, \ell\}$ (in particular, this implies that $e$ and $f$ do not share any vertices). The following result allows us to identify the co$(e_{i}, e_{j})$ that are actually edges of $\P_G$ using the notion of cycle-compatibility.\\

\begin{Lemma}[\cite{OH}]\label{edgeofedgepolytope}
Let $e$ and $f$ be edges of $G$ with $e\neq f$. Then {\rm co}$(e, f)$ is an edge of $\P_G$ if and only if $e$ and $f$ are not cycle-compatible.
\end{Lemma} 

Since the only integer points of edge polytopes are its vertices, the condition that $\P_G\cap \H^{(+)}$ and $\P_G\cap \H^{(-)}$ are 
integral is equivalent to the following: 
There exist a hyperplane $\H$ which is not a supporting hyperplane of $\P_G$ and such that for each edge $E$ of $\P_G$ with 
$\H \cap E \neq \emptyset$ it follows that $E \subset \H$ or $\H \cap E$ is an end point of $E$. 
This is by Lemma~\ref{edgeofedgepolytope}, equivalent to the following: for any pair of edges $e, f\in E(G)$ 
such that $\rho(e)\in \H^{(+)}\cap\partial \P$ and $\rho(f)\in \H^{(-)}\cap\partial \P$, $e$ and $f$ are cycle-compatible. \\

\begin{Proposition}[\cite{HLZ}]\label{hyperplaneform}
Let $G$ be a finite connected simple graph on $[d]$ and suppose that $\P_G \subset \RR^{d}$ is decomposable by $\H$. Then we can 
restrict attention to $\H$ of the following form:
\begin{equation}
\H =  \left\{ (x_1,\ldots ,x_d) \in \RR^d : \displaystyle\sum_{i=1}^d a_{i}x_{i}=0\   \right\}\nonumber
\end{equation}
where $a_{i}\in\{0, 1, -1\}$
\end{Proposition}
Proposition~\ref{hyperplaneform} allows us to assume that $\H^{(+)}$ contains points $(x_{1}, \ldots , x_{n})$ where $\sum a_{i}x_{i} \geq 0$ and $\H^{(-)}$ contains points $(x_{1}, \ldots , x_{n})$ where $\sum a_{i}x_{i} \leq 0$. 
For $(i, j)\in E(G)$, let the sign of $(i, j)$ be the sign of $a_{i}+a_{j}$
, the signature of $(i, j)$ be $\{a_{i}, a_{j}\}$ and the weight of vertex $i$ be $a_{i}$. These notations enable us to call an edge $(i, j)$ ``positive", ``negative" or ``zero", corresponding to whether the associated vertex $\rho((i, j))\in \P_G$ is in $\H^{(+)}\setminus\H$, $\H^{(-)}\setminus\H$, or $\H$. 

In section$2$, we will repeatedly use the following Propositions~\ref{dim} and~\ref{type}:

\begin{Proposition}[\cite{OH}]\label{dim}
Let $G$ be a finite simple graph on $[d]$. Then, ${\rm dim} \P_G = d - r - 1$,
where $r$ is the number of bipartite connected components of $G$.
\end{Proposition}

\begin{Proposition}[\cite{HLZ}]\label{type}
Suppose $G$ is decomposable. Then we must have at least one positive edge and at least one negative edge, 
and we can assume one of the following two cases for the vertices of $G$:
\begin{enumerate}
\item[(I)] There are no vertices with weight $0$. All positive edges have signature $\{1, 1\}$ and all negative edges have signature $\{-1, -1\}$. 
\item[(II)] There is at least one vertex with weight $0$. All positive edges have signature $\{1, 0\}$ and all negative edges have signature $\{-1, 0\}$.
\end{enumerate} 
\end{Proposition}

We call edge polytope $\P_G$  is {\em type I {\rm (}or type II{\rm )} decomposable} if there exist a separating hyperplane $\H$ 
satisfying condition (I) (or (II)) in Proposition~\ref{type}.
We say $G$ is decomposable if edge polytope $\P_G$ of $G$ is decomposable.
\begin{Example}
Following graph $G$ is type I and type II decomposable. \\
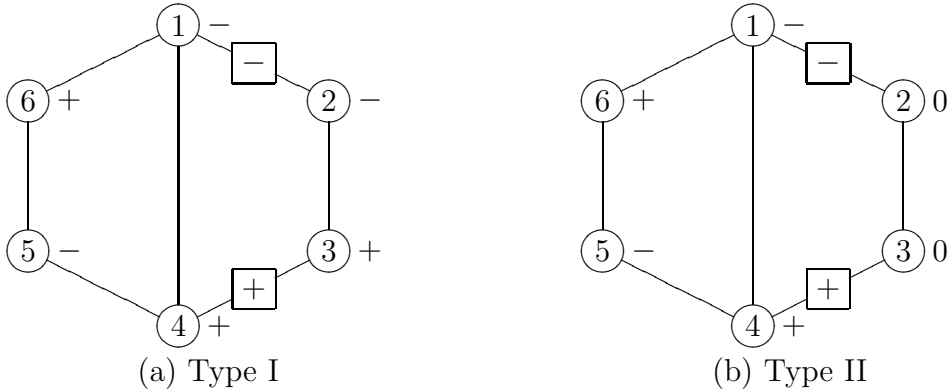
\begin{figure}[htbp]
\begin{tabular}{c}
\begin{minipage}{0.5\hsize}
\begin{xy}
\ar@{-} (20,40) *++!L{\ -} *+{1} *\cir<3mm>{}="1";
(20,0) *++!L{\ +} *+{4} *\cir<3mm>{}="4"
\ar@{-} "4"; (30, 5) *+[F]{+}="+1"
\ar@{-} "+1"; (40,10) *++!L{\ +} *+{3} *\cir<3mm>{}="3"
\ar@{-} "1";  (30,35) *+[F]{-}="-1"
\ar@{-} "-1"; (40,30) *++!L{\ -} *+{2} *\cir<3mm>{}="2"
\ar@{-} "2"; "3"
\ar@{-} "4"; (0,10) *++!L{\ -} *+{5} *\cir<3mm>{}="5"
\ar@{-} "1"; (0,30) *++!L{\ +} *+{6} *\cir<3mm>{}="6"
\ar@{-} "5"; "6"
\end{xy}
\hspace{1.6cm} (a) Type I
\end{minipage}

\begin{minipage}{0.5\hsize}
\begin{xy}
\ar@{-} (20,40) *++!L{\ -} *+{1} *\cir<3mm>{}="1";
(20,0) *++!L{\ +} *+{4} *\cir<3mm>{}="4"
\ar@{-} "4"; (30, 5) *+[F]{+}="+1"
\ar@{-} "+1"; (40,10) *++!L{\ 0} *+{3} *\cir<3mm>{}="3"
\ar@{-} "1";  (30, 35) *+[F]{-}="-1"
\ar@{-} "-1"; (40,30) *++!L{\ 0} *+{2} *\cir<3mm>{}="2"
\ar@{-} "2"; "3"
\ar@{-} "4"; (0,10) *++!L{\ -} *+{5} *\cir<3mm>{}="5"
\ar@{-} "1"; (0,30) *++!L{\ +} *+{6} *\cir<3mm>{}="6"
\ar@{-} "5"; "6"
\end{xy}
\hspace{1.6cm} (b) Type II
\end{minipage}
\end{tabular}
\caption{Type I and II decompositions}\label{hex}
\end{figure}

The type I decomposition of $\P_G$ is given by the separating hyperplane $\H :  -x_{1}-x_{2}+x_{3}+x_{4}-x_{5}+x_{6}=0$. The type II decomposition is given by the separating hyperplane $\H : -x_{1}+x_{4}-x_{5}+x_{6}=0$.
\end{Example}
Following the results by Hibi, Li and Zhang, we study decomposable edge polytopes.
\begin{Proposition}
Let $G$ be a finite simple graph and $H_1,\ldots H_n$ its connected components with $G = \bigsqcup_{i=1}^n H_i$. Edge polytope $\P_G$ is decomposable if and only if there exists a connected component $H_j$ such that $\P_{H_j}$ is decomposable.
\end{Proposition}
\begin{proof}
(``If'') We show this part by giving sign arrangements of separating hyperplanes. If $G$ satisfies (i), 
then we have a sign arrangement of $H_j$. we may set all the signs of the rest vertices 0. 
Then this sign arrangement implies that $G$ is decomposable.\\
(``Only if'') Suppose $G$ is decomposable and a separating hyperplane of $\P_G$ is given. First we show that $G$ has exactly one connected component which has a non-zero edge. Since $G$ is decomposable,
we have at least one positive edge and at least one negative edge.
If we have positive edge $e$ and negative edge $f$ in different connected components, then by applying Lemma~\ref{edgeofedgepolytope} $\rm{co}(e,f)$ is an edge of $\P_G$. The separating hyperplane cuts $co(e,f)$ into positive and negative part, a contradiction. 
Therefore we must have all the positive and negative edges in the same connected component.
Assume $H_j$ has positive and negative edges.
Remark that Proposition \ref{hyperplaneform} and \ref{type} was originally given for connected graphs
in \cite{HLZ} but we may apply these propositions to each connected components.
For each connected component $H_i$ with $i\neq j$, there are no non-zero edges. We set the weights of all vertices of $H_i$ 0.
We have to show cycle compatability of positive and negative edges in $H_j$. In fact, it easily follows from decomposability of $G$.
\end{proof}

This Proposition means that studying decomposability of connected graphs is also important to
discuss decomposability of disconnected graphs. So we study decomposability of connected graphs from now.
\if
By using the fact; for any complete multipartite graph $G$, every vertex disjoint two edges are cycle compatible, we can say the following.

\begin{Lemma}\label{ccc}
Complete multipartite graphs are type I decomposable if and only if there exist vertex disjoint two edges.
\end{Lemma}
\begin{proof}
The "only if'' part is known by using \cite[Theorem~2.2]{HLZ}.
Let $e, f$ be vertex disjoint two edges. Assume that $e$ is a positive edge with weight $\{1, 1\}$, $f$ is a negative edge with weight $\{-1, -1\}$. We can set weights either $1$ or $-1$ for all the other vertices. In type I sign arrangement, we do not have a pair of positive and a negative edge that has a common vertex. Thus every pair of positive and negative edges are cycle-compatible, as desired. 
\end{proof}
\fi

In the following example, we use an easy operation. We consider graph $G$ and path of length 3 $P_{3}=(x_{0}, x_{1}, x_{2}, x_{3})$. for any edge $\{i, j\} \in E(G)$, we join the path as  $x_{0}=i$, $x_{3}=j$, $x_{1}\neq k$, $x_{2}\neq k$ $(k \in V(G))$ and get the new graph. We call this operation, {\em attach} a 4-cycle to $G$ at edge $\{i,j\}$.

\begin{Example}
Let $G$ be a graph given in Figure~\ref{box2} (a).
After attaching a 4-cycle to $G$ at edge $\{2,3\}$, we have the graph Figure~\ref{box2} (b). 

\begin{figure}[htbp]
\begin{tabular}{c}
\begin{minipage}{0.33\hsize}
\begin{xy}
\ar@{-} (20,40) *+{1} *\cir<3mm>{}="1"; (20,0) *+{4} *\cir<3mm>{}="4"
\ar@{-} "4"; (40,10) *+{3} *\cir<3mm>{}="3"
\ar@{-} "1"; (40,30) *+{2} *\cir<3mm>{}="2"
\ar@{-} "2"; "3"
\ar@{-} "4"; (0,10) *+{5} *\cir<3mm>{}="5"
\ar@{-} "1"; (0,30) *+{6} *\cir<3mm>{}="6"
\ar@{-} "5"; "6"
\end{xy}
\hspace{1.6cm} (a)
\end{minipage}
\begin{minipage}{0.33\hsize}
\begin{xy}
\ar@{-} (20,40) *+{1} *\cir<3mm>{}="1"; (20,0) *+{4} *\cir<3mm>{}="4"
\ar@{-} "4"; (40,10) *+{3} *\cir<3mm>{}="3"
\ar@{-} "1"; (40,30) *+{2} *\cir<3mm>{}="2"
\ar@{-} "2"; "3"
\ar@{-} "4"; (0,10) *+{5} *\cir<3mm>{}="5"
\ar@{-} "1"; (0,30) *+{6} *\cir<3mm>{}="6"
\ar@{-} "5"; "6"
\ar@{-} "2" ; (60,30) *+{x_{1}} *\cir<3mm>{}="x_{1}"
\ar@{-} "x_{1}"; (60,10) *+{x_{2}} *\cir<3mm>{}="x_{2}"
\ar@{-} "x_{2}"; "3"
\end{xy}
\hspace{1.6cm} (b)
\end{minipage}
\end{tabular}
\caption{Attaching 4-cycle}\label{box2}
\end{figure}
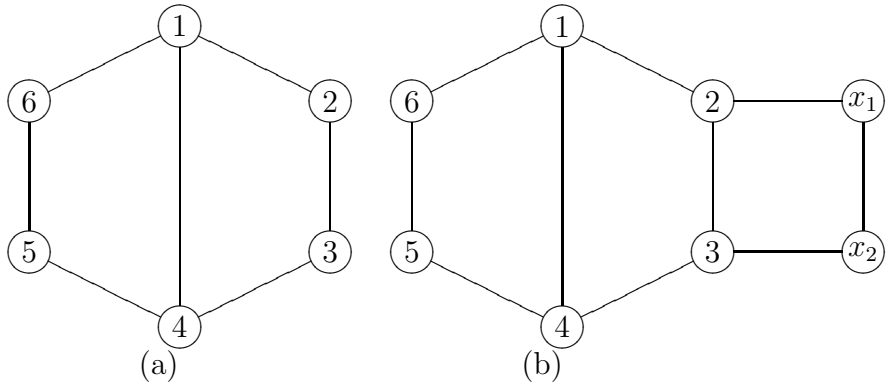
\end{Example}

\begin{Proposition}\label{box}
Let $G$ be a finite connected graph.
Let $G'$ be a graph obtained by attaching a 4-cycle to any edge of $G$.
Then, $P_{G'}$ is type II decomposable regardless of which edge we choose.
\end{Proposition}
\begin{proof}
Let $i,j$ be the vertices of $G'$ with $i \neq j$ which are not contained in V(G).
We set the weight $1$ for $i$  and $-1$ for $j$. We set the weight $0$ to all the other vertices in $V(G')$. Then the unique pair of the positive and negative edges are  cycle-compatible.
\end{proof}
We call the graph in Figure~\ref{tripan} (a) {\em tri-pan}. 
We consider two tri-pans, and name the vertices as Figure~\ref{tripan} (b) and join two tri-pans as $x=x'$, $y_{1}=y_{1}'$, then we get the following graph in Figure~\ref{tripan} (c).

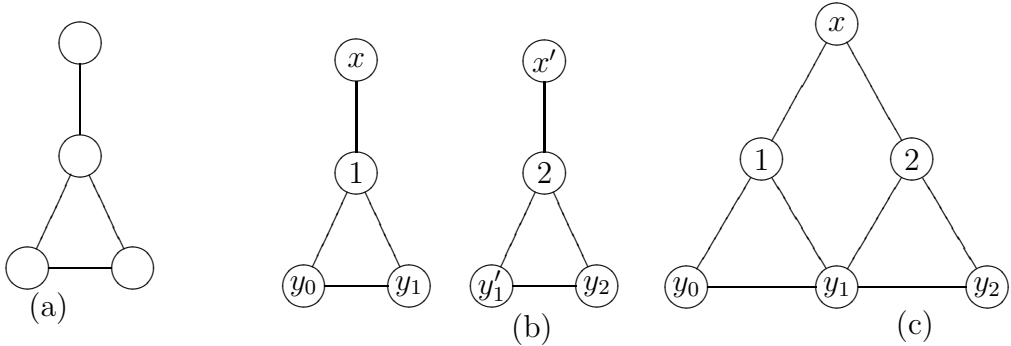
\begin{figure}[htbp]
\begin{tabular}{c}
\begin{minipage}{0.33\hsize}
\begin{xy}
\ar@{} (0,30) *+{} ; (23,30)  *\cir<3mm>{}="x"
\ar@{-} "x"; (23,15)  *\cir<3mm>{}="1"
\ar@{-} "1"; (16,0)  *\cir<3mm>{}="y_{0}"
\ar@{-} "1"; (30,0)  *\cir<3mm>{}="y_{1}"
\ar@{-} "y_{0}"; "y_{1}"
\end{xy}
\hspace{1.6cm} (a)
\end{minipage}
\begin{minipage}{0.33\hsize}
\begin{center}
\begin{xy}
\ar@{-} (7,30) *+{x} *\cir<3mm>{}="x"; (7,15) *+{1} *\cir<3mm>{}="1"
\ar@{-} "1"; (0,0) *+{y_{0}} *\cir<3mm>{}="y_{0}"
\ar@{-} "1"; (14,0) *+{y_{1}} *\cir<3mm>{}="y_{1}"
\ar@{-} "y_{0}"; "y_{1}"
\ar@{-} (32,30) *+{x'} *\cir<3mm>{}="x"; (32,15) *+{2} *\cir<3mm>{}="2"
\ar@{-} "2"; (25,0) *+{y_{1}'} *\cir<3mm>{}="y_{1}'"
\ar@{-} "2"; (39,0) *+{y_{2}} *\cir<3mm>{}="y_{2}"
\ar@{-} "y_{1}'"; "y_{2}"
\ar@{} (7,38) ; "x"
\end{xy}
\hspace{1.6cm} (b)
\end{center}
\end{minipage}
\begin{minipage}{0.33\hsize}
\begin{center}
\begin{xy}
\ar@{-} (40,35) *+{x} *\cir<3mm>{}="x"; (30,17) *+{1} *\cir<3mm>{}="1"
\ar@{-} "1"; (20,0) *+{y_{0}} *\cir<3mm>{}="y_{0}"
\ar@{-} "1"; (40,0) *+{y_{1}} *\cir<3mm>{}="y_{1}"
\ar@{-} "y_{0}"; "y_{1}"
\ar@{-} "x"; (50,17) *+{2} *\cir<3mm>{}="2"
\ar@{-} "2"; "y_{1}"
\ar@{-} "2"; (60,0) *+{y_{2}} *\cir<3mm>{}="y_{2}"
\ar@{-} "y_{1}"; "y_{2}"
\end{xy}
\hspace{1.6cm} (c)
\end{center}
\end{minipage}
\end{tabular}
\caption{Joined tri-pan}\label{tripan}
\end{figure}
We call the graph in Figure~\ref{tripan} (c)
Similarly, we define $n$-joined tri-pan $T(n)$. For example, Figure~\ref{ntripan} is $T(5)$.
\begin{figure}[htbp]
\begin{center}
\begin{xy}
\ar@{-} (225,25)  *\cir<3mm>{}="x"; (185,13) *\cir<3mm>{}="1"
\ar@{-} "x"; (205,13) *\cir<3mm>{}="2"
\ar@{-} "x"; (225,13) *\cir<3mm>{}="k-1"
\ar@{-} "x"; (245,13) *\cir<3mm>{}="k"
\ar@{-} "x"; (265,13) *\cir<3mm>{}="5"
\ar@{-} "1"; (175,0) *\cir<3mm>{}="y_{0}"
\ar@{-} "1"; (195,0) *\cir<3mm>{}="y_{1}"
\ar@{-} "2"; "y_{1}"
\ar@{-} "2"; (215,0) *\cir<3mm>{}="y_{2}"
\ar@{-} "k-1"; "y_{2}"
\ar@{-} "k-1"; (235,0) *\cir<3mm>{}="y_{k-1}"
\ar@{-} "k"; "y_{k-1}"
\ar@{-} "k"; (255,0) *\cir<3mm>{}="y_{4}"
\ar@{-} "5"; "y_{4}"
\ar@{-} "5"; (275,0) *\cir<3mm>{}="y_{5}"
\ar@{-} "y_{0}"; "y_{1}"
\ar@{-} "y_{2}"; "y_{1}"
\ar@{-} "y_{2}"; "y_{k-1}"
\ar@{-} "y_{k-1}"; "y_{4}"
\ar@{-} "y_{4}"; "y_{5}"
\end{xy}
\end{center}
\caption{$5$-joined tri-pan}\label{ntripan}
\end{figure}
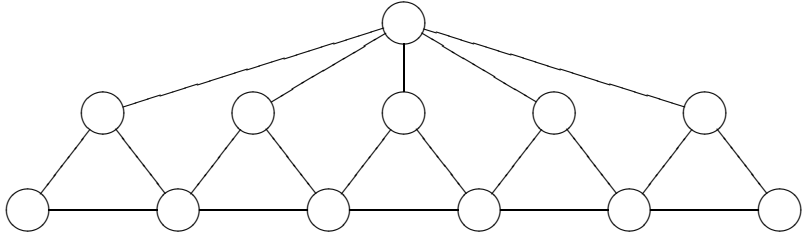
Let $N_G(k)$ denote that neighbor set of k in $G$
\bigskip
\begin{Proposition}\label{indecomposable}
Suppose that $T(n)$ is a $n$-joined tri-pan. Then $\P_{T(n)}$ is indecomposable.
\end{Proposition}
\begin{proof}
$\P_{T(1)}$ is indecomposable because there is no 4-cycle in $T(1)$.
Suppose that $n\ge2$.
First, we check type I decomposablity.
Assume that $\P_{T(n)}$ is type I decomposable. Fix an sign arrangement of $G$.
Since $\P_{T(n)}$ is decomposable, we have at least one 4 cycle with positive and nagative edge in $T(n)$ that satisfies cycle-compatability.
Assume that the 4 cycle is $(x,k-1,y_{k-1},k)$ in Figure~\ref{proof}.
Without losing generality, we may assume weights of vertices $x,k-1,y_{k-1},k$ are $1, 1, -1, -1 $ or $-1, -1, 1 , 1$.
In the first case, take a look at the weight of $y_{k-2}$. It is not possible to suppose the weight of $y_{k-2}$ is $1$ because then a positive edge $\{y_{k-2}, k-1 \}$ and a negative edge ${y_k-1}, k$ do not satisfy cycle compatability.
On the other hand, we can not set the weight $-1$ for $y_{k-2}$ because again from cycle compatability of positive and negative edges. 
Similarly, if we are in the second case, we can not choose the 4 cycle without losing cycle compatability.
Next, we check type II decomposability.
Suppose that $T(n)$ is type II decomposable.
We have a 4-cycle $(x, k-1 , y_{k-1}, k)$ with weights $\{1,-1,0,0\}$ 
The weight of the edge $\{k-1, y_{k-1} \}$ must be $\{0,0\}$. Otherwise, We can not set the weight of the vertex of $N_{T(n)}(k-1)\cap N_{T(n)}(y_{k-1})$ because one of the following occurs.
\begin{itemize}
\item[(i)] A weight of a edge is $\{1,1\}$ or $\{-1,-1 \}$.
\item[(ii)] One of the positive edge and one of the negative edge have a common vertex.
\end{itemize}
Assume that the weight of $\{k-1 ,y_{k-1}\}$ is $\{0,0\}$ then the weight of $k$ is either $1$ or $-1$.
Let $N(k)\cap N(y_{k-1}) = y_{k}$. Then, we can not set the weight for $y_k$ because again one of the above occurs.

\begin{figure}[h]
\begin{xy}
\ar@{-} (225,25) *+{x} *\cir<3mm>{}="x"; (185,13) *\cir<3mm>{}="1"
\ar@{-} "x"; (205,13) *\cir<3mm>{}="2"
\ar@{-} "x"; (225,13) *+{k-1} *\cir<5mm>{}="k-1"^{+}
\ar@{-} "x"; (245,13) *+{k} *\cir<5mm>{}="k"_{0}
\ar@{-} "x"; (265,13) *\cir<3mm>{}="5"
\ar@{-} "1"; (175,0) *\cir<3mm>{}="y_{0}"
\ar@{-} "1"; (195,0) *\cir<3mm>{}="y_{1}"
\ar@{-} "2"; "y_{1}"
\ar@{-} "2"; (215,0) *+{y_{k-2}} *\cir<5mm>{}="y_{2}"
\ar@{-} "k-1"; "y_{2}"
\ar@{-} "k-1"; (235,0) *+{y_{k-1}} *\cir<5mm>{}="y_{k-1}"^{0}
\ar@{-} "k"; "y_{k-1}"^{-}
\ar@{-} "k"; (255,0) *\cir<3mm>{}="y_{4}"
\ar@{-} "5"; "y_{4}"
\ar@{-} "5"; (275,0) *\cir<3mm>{}="y_{5}"
\ar@{-} "y_{0}"; "y_{1}"
\ar@{-} "y_{2}"; "y_{1}"
\ar@{-} "y_{2}"; "y_{k-1}"
\ar@{-} "y_{k-1}"; "y_{4}"
\ar@{-} "y_{4}"; "y_{5}"
\end{xy}
\caption{}\label{proof}
\end{figure}
\end{proof}
\begin{Theorem}
There exist infinite number of graphs that is both type I and type II decomposable.
Similary, one has type I but not type II, type II but not type I and neither type I nor type II.

\begin{itemize}
\item[(i)]Let $G$ be a complete multipartite graphs with at least $4$ vertices. Let $G'$ be a graph 
obtained by attaching a 4-cycle at any edge of $G$. Then $\P_{G'}$ is type I and II decomposable.
\item[(ii)]Let $G$ be a complete graph with at least $4$ vertices. Then $\P_G$ is type I decomposable and type II indecomposable.
\item[(iii)]Let T(n) be a $n$-joined tri-pan with $n\ge2$ and $G$ be a graph obtained by attaching a 4-cycle at any edge of $T(n)$. Then $\P_G$ is type II decomposable and type I indecomposable.
\item[(iv)]Let $T(n)$ be a $n$-joined tri-pan. Then $\P_{T(n)}$ is indecomposable.
\end{itemize}
\end{Theorem}
\begin{proof}
We know by Proposition 1.7 that  the edge polytopes of (i) and (iii) are type II decomposable.
Also, (iv) is indecomposable as we shown in Proposition 1.8. 
We have to show the rest.
First we show (i) is type I decomposable. In the complete multipartite graphs every pair of vertex disjoint edges are cycle-compatible.
We set the weight $1$ for 2 vertices of the chosen edge $e$ which we attached the 4-cycle and set the weight $-1$ for all the other vertices.
Then this sign arrangement implies type I decomposability since $e$ and all the other edges that has no common vertex with $e$ are cycle-compatible. 

Next we show (ii) is type I decomposable. Let $G = K_d$ with vertex set $[d]$. We can set the weight $1$ to vertex $1$ and $2$ and $-1$ to the others. Since all distinct pair of edges are cycle-compatible, complete graphs are type I decomposable as required.
(ii) is not type II decomposable. Let $G = K_d$ with vertex set $[d]$. Suppose that $\P_G$ is decomposable. then there exist at least one pair of edges with positive and negative sign, say $e = (i,j)$ and $f = (k,l)$. By symmetry of complete graphs we may assume that weights of $i,j,k$ and $l$ are $1,0,-1$ and $0$ We can not make this sign arrangement because a positive edge $(i,j)$ and a negative edge $(k,l)$ share a vertex.

(iii) is not type I decomposable. We can not set the weights  for the shape of $T(2)$ without breaking cycle-compatability of positive and negative edges, as we see in the proof of Proposition~\ref{indecomposable}.

\end{proof}

\section{Description of decomposable edge polytopes}
In this section, we discuss about decomposable edge polytopes in terms of underlying graphs.
Let $V'$ be the subset of $V$, $G[V']$ the induced subgraph of $G$ with vertex set $V'$ and $N(V')$ be the neighbour set of $V'$.
We call a family of vertex set $\{V_1,\ldots, V_n\}$, $V_i \in V$ is a {\em vertex partition} of $V$ if it satisfies
$\bigcup �V_i = V$ and $V_i \cap V_j = \emptyset$ for any $i \neq j$.
We call a graph $G$ is {\em empty} if there exist no edge in $G$
\begin{Lemma}\label{0condition}
$G$ is bipartite if and only if there exist a hyperplane $\H = \sum_i a_i x_i$ with $|a_1|= \ldots = |a_d| =1 $ 
such that all signatures of edges of  $G$ are $0$.
\end{Lemma}
\begin{proof}
Suppose $V_1 \cup V_2$ is the bipartition of $G$.
We set weight $1$ for every vertex in $V_1$ and $-1$ for every vertex in $V_2$.
It has proved because $G[V_1]$ and $G[V_2]$ are not empty.

Conversely suppose we have an odd cycle $\{v_1,v_2,\ldots , v_{2n+1} ,v_1 \}$ 
in $G$ where $n$ is a positive integer.
If weight of $v_1$ is $1$, then we have to set weight of $v_2$ to be $-1$ and $v_3$ to be $1$.
By continuing this operation, we get that weight of vertex $v_{2n+1}$ is $1$. This  contradicts that there is no positive edge.
Similarly, a contradiction occur when we start from weight of vertex $v_1$ with $-1$.  
\end{proof}

\begin{Proposition} \label{theorem}
Suppose that $G$ is a finite connected simple graph on $[d]$.  
\begin{itemize}
\item[(i)] Edge polytope $\P_G$ is type I decomposable if and only if there exist a vertex partition $\{V_{+},V_{-}\}$ of $G$
such that $G_+ = G[V_+]$ and $G_- = G[V_-]$ 
are not empty graph and every pair of edges 
($e \in E(G_+)$ , $ f \in E(G_-)$) is cycle-compatible.
\item[(ii)] Edge polytope $\P_G$ is type II decomposable if and only if there exist a vertex partition $\{V_1, V_2, V_3, V_4, V_5\}$ of $G$ that satisfies the follwoing conditions.
Let $E_{i,j}$ denote the set of edges between $V_i$ and $V_j$. \\
\rm{(1)} $G[V_1 \cup V_2]$ is a bipartite graph with a bipartition $V_1, V_2$.
\rm{(2)} $|E_{1,4}| \geq  1$, $|E_{2,3}| \geq 1$ and every pair of edges $e$ and $f$ $(e \in E_{1,4}$ , $f \in E_{2,3})$ is cycle-compatible.
\rm{(3)} $|E_{1,3}| =|E_{1,5}|= |E_{2,4}| =|E_{2,5}|=0$. 
\end{itemize}
\end{Proposition}

\begin{proof}
(i)(if part) Suppose that $G$ is type I decomposable with vertex partition $V_1, V_2$. We set weights of vertices in vertex set $V_1$ to be $1$ and $V_2$ to be $-1$. Then it follows $\P_G$ is type I decomposable.
(only if part) Suppose that $G$ is a decomposable graph with a type I decomposition.
Fix a sign arrangement on $G$ that gives type I decomposition.
By Proposition~\ref{type} we may say that every weight is $1$ or $-1$.
Let $V_+$(resp. $V_-$) denote the set of  $1$-weighted (resp. ($-1$)-weighted) vertices.
Since $G$ is type I decomposable, we must have at least one positive edge in $G[V_+]$ and at least one negative edge in $G[V_-]$.
We know that every pair of positive and negative edges are cycle compatible. Thus $G$ satisfies (1).

(ii) (if part) Suppose we have such a vertex partition. We set weights of vertices in $V_1$ to be $1$, $V_2$ to be $-1$, $V_3, V_4, V_5$ to be $0$. Then it is obvious that $G$ is type II.
(only if part) Suppose that $G$ is type II decomposable. Fix a sign arrangement that gives type II decomposition. Let $V_+, V_-, V_0$ be the set of, namely $1, -1, 0$-weighted vertices.
Let $V_1 = V_+$, $V_2 = V_-$.
Since this sign arrangement gives type II decomposition, we do not have any edge with weight $(1,1)$ or $(-1,-1)$. Therefore, condition(1) is satisfied. We set $V_4= N(V_+) \cap V_0$ and $V_3 = V (N_-) \cap V_0$ and $V_5 = V \setminus (V_1\cup V_2\cup V_3 \cup V_4)$.
Note that $N(V_+) \cap N(V_-) = \emptyset$. Thus, condition (3) is satisfied. 
Also condition(2) is satisfied because every pair of positive and negative edges are cycle-compatible.
\end{proof}
We say a subgraph $G'$ of $G$ is a {\em spanning} subgraph of $G$ if vertex set of $G'$ is $[d]$ with no isolated vertex in $G'$.

Recall that any integral subpolytope of $\P_G$ is again an edge polytope.
For separating hyperplane $\H$ of $\P_G$, 
let $G_0$ be the graph which satisfies $\P_{G_0} = \P_G \cap \H$.
Then the edge set of $G_0$ is the set of zero edges of $G$ defined by separationg hyperplane $\H$.
Remark that $G_0$ is not necessarily a connected graph.
\begin{Proposition}\label{Nantheorem}
Let $G$ be a connected bipartite graph. $\P_G$ is type I decomposable if and only if $\P_G$ is type II decomposable.
\end{Proposition}
\begin{proof}
``if" part is given in \cite{HLZ}.
Now we have to show ``only if" part. Suppose that $\P_G$ is type I decomposable.
Let $V_1,V_2$ be a bipartition of $G$, put the vertices of $V_1$ on the left and $V_2$ on the right. Since $\P_G$ is type I decomposable,
 let $H$ be a sign arrangement on $G$ that is type I. Then we can see that on both of the left side and the right side, there are
 both positive and negative signs. Now apply the following change to the sign arrangement $H$: change all the
 possitive vertices on the left to zeros, and change all the negative vertices on the right to zeros. We call this new sign arrangement
 $H'$. We claim that with this $H'$ makes $G$ type II decomposable. 
 First, it is clear that in $H'$, all the edges have the form $(0,1)$, $(0,-1)$ or $(0,0)$. Then it is left to show that any pair of positive edge and negative edge in $H'$ satisfies the cycle compatability. This is true because an edge in $H'$ is positive (or negative)
 if and only if that edge in $H$ is positive (or negative). Since any pair
 of positive edge and negative edge in $H$ satisfies cycle-compatability, so does in $H'$.
\end{proof}
\begin{Example}
For the following bipartite graph, we can give a type I sign arrangement as Figure~\ref{bipartite} (a).
We have two type II sign arrangement obtained by Proposition~\ref{Nantheorem}. Those are Figure~\ref{bipartite} (b) and (c).
\begin{figure}[htbp]
\begin{tabular}{c}
\begin{minipage}{0.33\hsize}
\begin{center}
\includegraphics[width = 4.5cm]{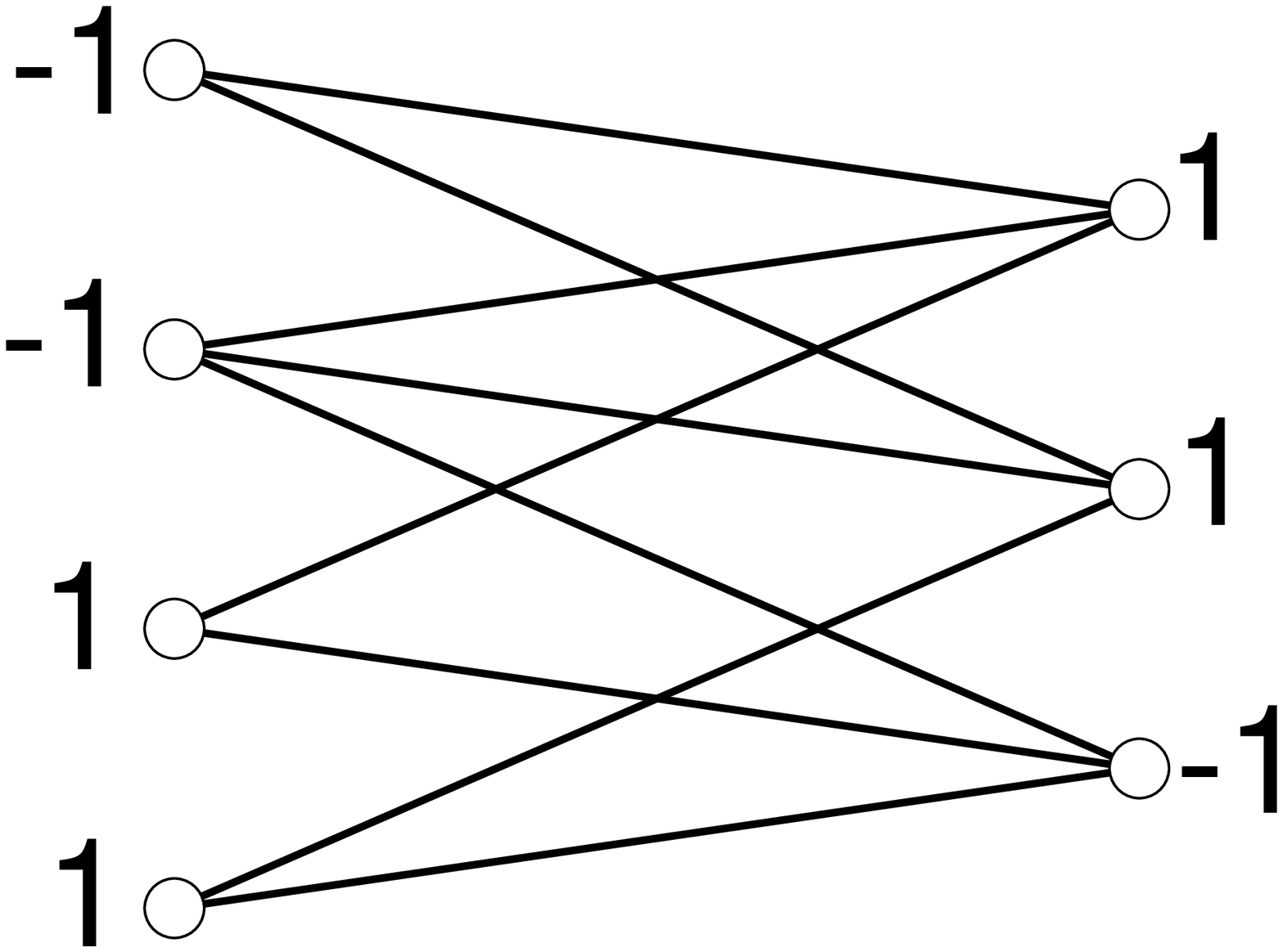}
\hspace{1.6cm} (a)
\end{center}
\end{minipage}

\begin{minipage}{0.33\hsize}
\begin{center}
\includegraphics[width = 4.5cm]{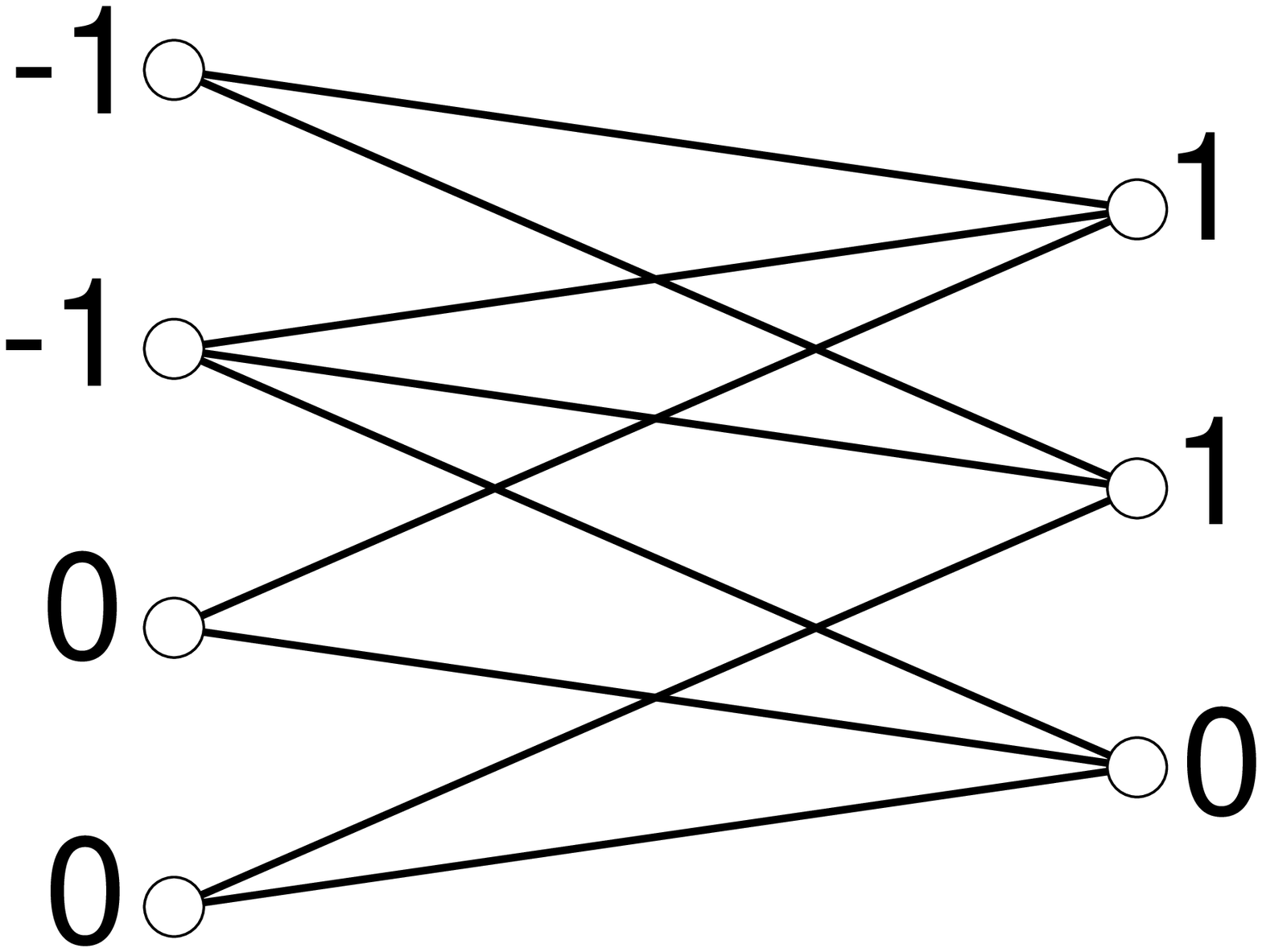}
\hspace{1.6cm} (b)
\end{center}
\end{minipage}

\begin{minipage}{0.33\hsize}
\begin{center}
\includegraphics[width = 4.5cm]{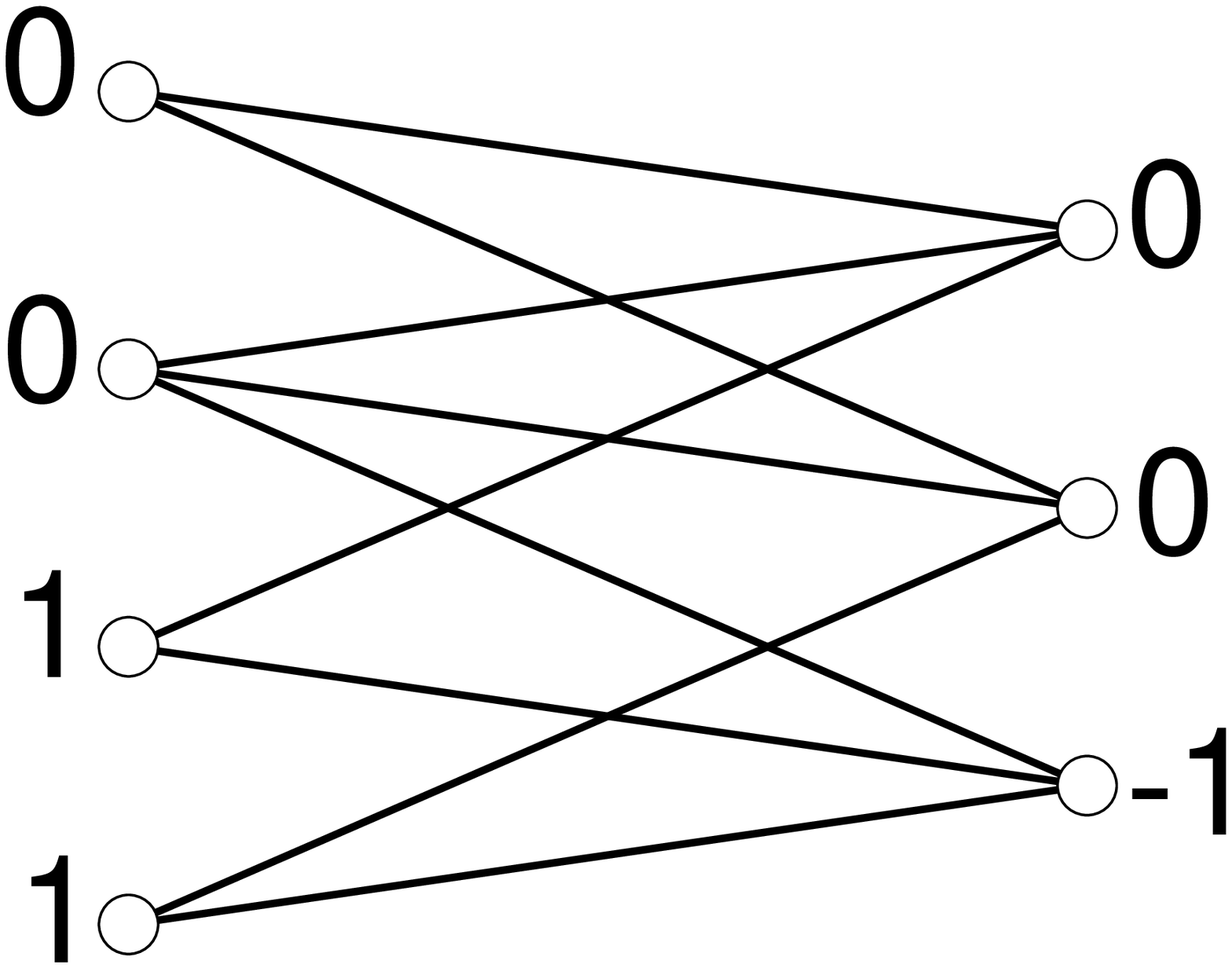}
\end{center}
\hspace{1.6cm} (c)
\end{minipage}
\end{tabular}
\caption{}\label{bipartite}
\end{figure}
\end{Example}
\newpage
As the corollary of Proposition~\ref{theorem} and
Proposition~\ref{Nantheorem}, we can describe decomposability of edge polytopes as the followings. 
\begin{Theorem}\label{charactorization}
Suppose that $G$ is a connected simple graph.
\begin{itemize}
\item[(a)] If $G$ has at least one odd cycle and is type I decomposable, then $G_0$ is a connected bipartite graph.
\item[(b)] If $G$ has at least one odd cycle and is type II decomposable, then $G_0$ consists of exactly $2$ connected components. One is bipartite and the other is not bipartite.
\item[(c)] If $G$ is bipartite and type I decomposable, then $G_0$ is consists of exactly $2$ connected components.
\end{itemize}
\end{Theorem}
\begin{proof}
In this proof we repeatedly use Proposition~\ref{dim} that claims dimension of edge polytope is reduced by the number of connected bipartite components.
 \begin{itemize}
\item[(a)] Let $G$ be a connected non-bipartite graph. Then ${\rm dim}(\P_G) = d - 1 $ because only one connected component is not bipartite.
It is known from Lemma~\ref{0condition} that $G_0$ is bipartite.  If $G_0$ is not connected, then we have at least $2$ connected bipartite components. On the other hand, we know that ${\rm dim} (\P_G) = d-1$ and hence $ {\rm dim} (\P_{G_0})= {\rm dim}(\H) = d-2 $. Thus, $G_0$ is connected.

\item[(b)] Since $G$ is type II decomposable, Let $V_1,\ldots V_5$ be a vertex  partition of $G$ given by Proposition~\ref{theorem}.
As we see in the proof of Proposition~\ref{theorem}, we obtain $ G_0 = G[V_1\cup V_2] \cup G[V_3 \cup V_4 \cup V_5]$.
Now we show that $G[V_1\cup V_2]$ is connected bipartite and $G[V_3\cup V_4\cup V_5]$ is connected non-bipartite.
$G[V_1 \cup V_2]$ is bipartite because $G[V_1]$ and $G[V_2]$ is empty graphs.
$G[V_1 \cup V_2]$ is connected, otherwise ${\rm dim}(\H) \le {\rm dim}(\P_G)-2$, it contradicts $\H$ is a separating hyperplane.
Suppose $G[V_3 \cup V_4\cup V_5]$ is bipartite. Then we have at least 2 bipartite connected components, again a contradiction. That implies $G[V_3 \cup V_4\cup V_5]$ is not bipartite.
We may assume that for any vertex $i \in V_3$ (or $V_4$), there exist a vertex $j \in V_2$ (or $V_1$) such that $\{i,j\}$ is an edge in $G$ and for any vertex. (If we have vertex which does not have neighbors in $V_2$(or $V_1$), we may move the vertex to $V_5$. Then new vertex partition again satisfies our condition.)
Since vertices in $V_3$ and $V_4$ are end points of positive or negative edges, 
$G[V_3\cup V_4]$ form a complete bipartite graph. Thus, $G[V_3 \cup V_4]$ is connected.
Our edge delesion from $G$ to $G_0$ is exactly the delesion of $E_{1,4}$ and $E_{2,3}$.
$G$ is connected and vertices in $V_5$ are not connected to $V_1$ and $V_2$. That implies for any vertex $v\in V_5$, there exist a vertex $u\in V_3 \cup V_4$ and a finite path from $v$ to $u$ in $G[V_3\cup V_4\cup V_5]$. Therefore, $G[V_3\cup V_4\cup V_5]$ is connected.

\item[(c)] Since $G$ is connected bipartite, we know ${\rm dim} (\P_G) = d-2 $ and hence ${\rm dim} (\H) =  d-3 $. Since $G_0$ is a subgraph of $G$, all the connected components are bipartite. Then $G_0$ consists of $2$ connected bipartite components.
\end{itemize}
\end{proof}

We are grateful to Prof. Alexander Engstr\"om and Prof. Patrik Nor\'en for their useful suggenstions on this research.

{}
\end{document}